\documentclass[a4paper,11pt,reqno]{amsart}

\usepackage[utf8]{inputenc}
\usepackage{amsfonts}
\usepackage{amsmath}
\usepackage{amssymb}
\usepackage{amsthm}
\usepackage{comment}
\usepackage{graphicx}
\usepackage{mathtools}
\usepackage{subcaption}
\usepackage{tikz}
\usepackage[margin=1in]{geometry}
\usepackage[numbers]{natbib}
\usepackage[colorlinks=true,linkcolor=blue,urlcolor=blue,citecolor =blue,anchorcolor=blue]{hyperref}

\usepackage[shortlabels]{enumitem}
\usepackage{bbm}

\usepackage{dsfont}

\colorlet{darkgreen}{green!50!black}

\usepackage{hyperref} 
\hypersetup{	
colorlinks=true,  
breaklinks=true,  
urlcolor= green, 
linkcolor= blue,	
citecolor=darkgreen,	
pdftitle={Dual skew symmetry}, 
} 

\usepackage{soul}

\allowdisplaybreaks
\definecolor{darkgreen}{RGB}{0,100,0}

\renewcommand{\geq}{\geqslant}
\renewcommand{\leq}{\leqslant}

\newtheorem*{definition}{Definition}
\newtheorem{theorem}{Theorem}
\newtheorem{assumption}{Assumption}
\newtheorem{lemma}[theorem]{Lemma}
\newtheorem{proposition}[theorem]{Proposition}

\newtheorem{remark}{Remark}

\usepackage{soul}
\usepackage[normalem]{ulem}

\DeclareMathOperator{\diag}{diag}

\author{S\MakeLowercase{andro} F\MakeLowercase{ranceschi and} K\MakeLowercase{ilian} R\MakeLowercase{aschel}}
        \address{Laboratoire SAMOVAR, T\'el\'ecom SudParis, Institut Polytechnique de Paris, 91120 Palaiseau, France} \email{sandro.franceschi@telecom-sudparis.eu}

\address{Laboratoire Angevin de Recherche en Math\'ematiques, CNRS, Universit\'e d'Angers, 49000 Angers, France
        } \email{raschel@math.cnrs.fr}
      \thanks{This project has received funding from the ANR RESYST (ANR-22-CE40) and from the European Research Council (ERC) under the European Union's Horizon 2020 research and innovation programme under the Grant Agreement No.\ 759702.}

\date{\today}

\begin{document}
\title[Dual skew symmetry in an orthant]{A\MakeLowercase{ dual skew symmetry for transient reflected} {B}\MakeLowercase{rownian motion in an orthant}}


\begin{abstract}
We introduce a transient reflected Brownian motion in a multidimensional orthant, which is either absorbed at the apex of the cone or escapes to infinity. We address the question of computing the absorption probability, as a function of the starting point of the process. We provide a necessary and sufficient condition for the absorption probability to admit an exponential product form, namely, that the determinant of the reflection matrix is zero. We call this condition a dual skew symmetry. It recalls the famous skew symmetry introduced by Harrison \cite{Ha-78}, which characterizes the exponential stationary distributions in the recurrent case. The duality comes from that the partial differential equation satisfied by the absorption probability is dual to the one associated with the stationary distribution in the recurrent case.
\end{abstract}


\maketitle

\section{Introduction and main results}

\subsection*{Reflected Brownian motion in orthants}

Reflected Brownian motion (RBM) in orthants $\mathbb R_+^d$ is a fundamental stochastic process. Starting from the eighties, it has been studied in depth, with focuses on its definition and semimartingale properties \cite{VaWi-85,Wi-85b,Wi-95}, its recurrence or transience \cite{Wi-85a,HoRo-93,Ch-96,BrDaHa-10,Br-11}, the possible particular (e.g., product) form of its stationary distribution \cite{HaWi-87,DiMo-09}, the asymptotics of its stationary distribution \cite{HaHa-09,DaMi-11}, its Lyapunov functions \cite{DuWi-94,Sa-17}, its links with other stochastic processes \cite{LG-87,Du-04,Le-17}, its use to approximate large queuing networks \cite{Ha-78,Fo-84,BaFa-87}, numerical methods to compute the stationary distribution \cite{DaHa-92}, links with complex analysis \cite{Fo-84,BaFa-87,FrRa-19,BMEPFrHaRa-21}, PDEs \cite{HaRe-81a}, etc. The RBM is characterized by a covariance matrix $\Sigma$, a drift vector $\mu$ and a reflection matrix $R$. We will provide in Section~\ref{sec:absorption} a precise definition. While $\Sigma$ and $\mu$ correspond to the Brownian behavior of the process in the interior of the cone, the matrix $R$ describes how the process is reflected on the boundary faces of the orthant. In the semimartingale case, RBM admits a simple description using local times on orthant faces, see~\eqref{eq:RBM_semimart}.

\subsection*{Dimensions $1$ and $2$} 
The techniques to study RBM in an orthant very heavily depend on the dimension. In dimension $1$, RBM with zero drift in the positive half-line $\mathbb R_+$ is equal, in distribution, to the absolute value of a standard Brownian motion, via the classical Tanaka formula; if the drift is non-zero, the RBM in $\mathbb R_+$ is connected to the so-called bang-bang process \cite{Sh-81}. Most of the computations can be performed explicitly, using closed-form expressions for the transition kernel.

The case of dimension $2$ is historically the one which attracted the most of attention, and is now well understood. Thanks to a simple linear transform, RBM in a quadrant with covariance matrix is equivalent to RBM in a wedge with covariance identity, see~\cite[Appendix]{FrRa-19}. The very first question is to characterize the parameters of the RBM (opening of the cone $\beta$ and reflection angles $\delta,\varepsilon$, see Figure~\ref{fig:dim-2})\ leading to a semimartingale RBM, as then tools from stochastic calculus become available. The condition takes the form $\alpha<1$, see \cite{Wi-85a}, with
\begin{equation}
\label{eq:def_alpha}
   \alpha=\frac{\delta+\varepsilon-\pi}{\beta}.
\end{equation}
As a second step, conditions for transience and recurrence were derived, see \cite{HoRo-93,Wi-85a}. Complex analysis techniques prove to be quite efficient in dimension $2$, see \cite{BaFa-87,BMEPFrHaRa-21}. In particular, this method leads to explicit expressions for the Laplace transforms of quantities of interest (stationary distribution in the recurrent case \cite{FrRa-19,BMEPFrHaRa-21}, Green functions in the transient case \cite{Fr-20}, escape and absorption probabilities \cite{ErFrHu-20,FoFrIv-22}).

\subsection*{Higher dimension}
As opposed to the previous cases, the case of $d>2$ is much most mysterious. However, necessary and sufficient conditions for the process to be a semimartingale are known, and read as follows: denote the reflection matrix by
\begin{equation}
\label{eq:def_reflection_matrix}
R=\begin{pmatrix}
    1 & r_{12} & \dots & r_{1d} \\
    r_{21} & 1 & \dots & r_{2d} \\
    \vdots & \vdots & \ddots & \vdots \\
    r_{d1} & r_{d2} & \dots & 1
  \end{pmatrix}.
\end{equation}
The column vector
\begin{equation}
\label{eq:def_reflection_vector}
   R_j=\begin{pmatrix}
    r_{1j} \\
  \vdots  \\
    r_{d j} 
  \end{pmatrix}
\end{equation} 
represents the reflection vector on the orthant face $x_i=0$. Then the RBM is a semimartingale if and only if the matrix $R$ is completely-$\mathcal S$, in the following sense, see~\cite{ReWi-88,TaWi-93}. 

By definition, a principal sub-matrix of $R$ is any matrix of the form $(r_{ij})_{(i,j)\in I^2}$, where $I$ is a non-empty subset of $\{1,\ldots,d\}$, possibly equal to $\{1,\ldots,d\}$. If $x$ is a vector in $\mathbb R^d$, we will write $x> 0$ (resp.\ $x\geqslant 0$) to mean that all its coordinates are positive (resp.\ non-negative). We define $x<0$ and $x\leqslant 0$ in the same way. The definition extends to matrices. 
\begin{definition}[$\mathcal S$-matrix]
A square matrix $R$ is an $\mathcal{S}$-matrix if there exists $x \geqslant 0$ such that
$Rx > 0$. Moreover, $R$ is completely-$\mathcal{S}$ if all its principal sub-matrices are $\mathcal{S}$-matrices.
\end{definition}

Apart from the semimartingale property, very few is known about multidimensional RBM. In particular, necessary and sufficient conditions for transience or recurrence are not yet fully known in the general case, even though, under some additional hypothesis on $R$, some conditions are known \cite{HaWi-87bis,Ch-96}, or in dimension 3 \cite{ElBeYa-00,BrDaHa-10}. For example, if $R$ is assumed to be a non-singular $\mathcal{M}$-matrix (which means that $R$ is an $\mathcal{S}$-matrix whose off-diagonal entries are all non-positive), then $R^{-1}\mu <0$ is a necessary and sufficient condition for positive recurrence. Moreover, contrary to the two-dimensional case, no explicit expressions are available for quantities of interest such as the stationary distribution, in general.

\begin{figure}
\includegraphics[scale=0.7]{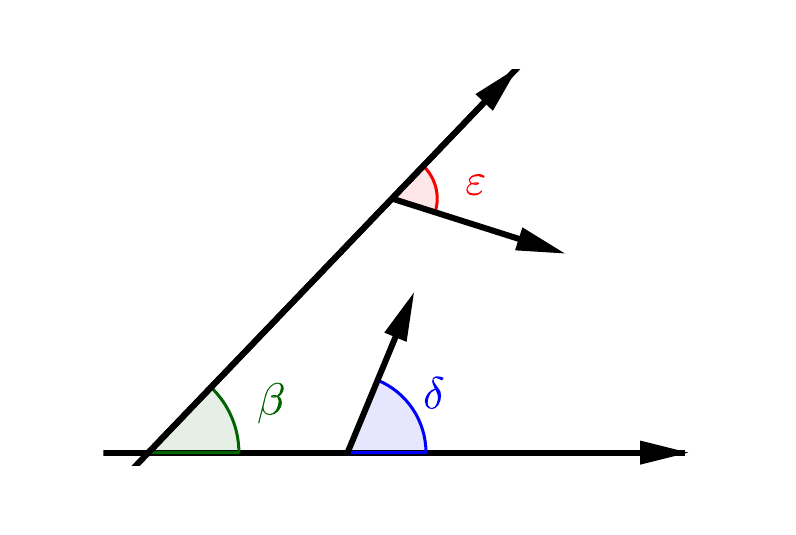} 
\caption{Wedge and angles of reflection.}
\label{fig:dim-2}
\end{figure}

\subsection*{The historical skew symmetry condition} 
The only notable and exceptional case, in which everything is known and behaves smoothly, is the so-called skew symmetric case, as discovered by Harrison \cite{Ha-78} in dimension $2$, and Harrison and Williams \cite{HaWi-87} in arbitrary dimension. They prove that the RBM stationary distribution has a remarkable product form
\begin{equation}
\label{eq:parameters_product_form}
   \pi(x_1,\ldots, x_d) = c_1\cdots c_d\exp(-c_1x_1-\cdots -c_dx_d)
\end{equation}
if and only if the following relation between the covariance and reflection matrices holds:
\begin{equation}
\label{eq:skew_symmetry}
   2\Sigma = R\cdot \diag \Sigma + \diag \Sigma \cdot R^\top.
\end{equation}
In the latter case, the stationary distribution admits the exponential product form given by \eqref{eq:parameters_product_form}, with parameters equal to
\begin{equation*}
   (c_1,\ldots,c_d)^\top = -2\cdot (\diag \Sigma)^{-1} \cdot R^{-1} \cdot \mu,
\end{equation*}
with $\mu$ denoting the drift vector. 
In dimension 2, if we translate this model from the quadrant to a wedge, condition \eqref{eq:skew_symmetry} is equivalent to $\alpha=0$, see \cite[Sec.~5.2]{FrRa-19} and our Figure~\ref{fig2:dim-2}.
Models having this skew symmetry are very popular, as they offer the possibility of computing the stationary distribution in closed-form. No generalization of the skew symmetry is known, except in dimension $2$, where according to \cite{DiMo-09}, the stationary distribution is a sum of $n\geq1$ exponential terms as in \eqref{eq:parameters_product_form} (with suitable normalization) if and only if $\alpha=-n$, where the parameter $\alpha$ is as in \eqref{eq:def_alpha}. The recent article \cite{BMEPFrHaRa-21} goes much further, generalizing again this result and finding new conditions on $\alpha$ to have simplified expressions of the density distribution.

The concept of skew symmetry has been explored in other cases than orthants, see for example \cite{Wi-87,OCOr-14}.

\begin{figure}
\includegraphics[scale=1.3]{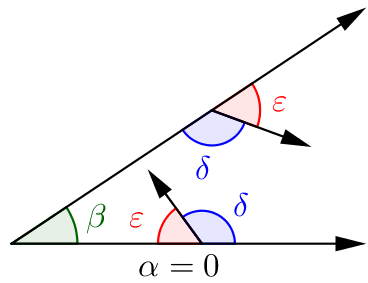} 
\hspace{1cm}
\includegraphics[scale=1.3]{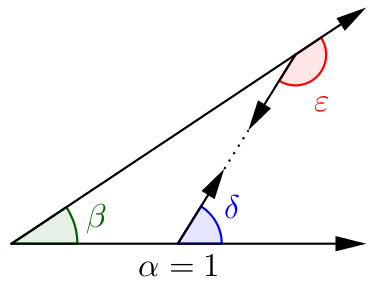} 
\caption{On the left, the standard skew symmetry condition in a wedge, corresponding to the condition $\alpha=0 $; on the right, the dual skew symmetry condition  $\alpha=1 $.}
\label{fig2:dim-2}
\end{figure}

\subsection*{Our approach and contributions}
In this paper, we will \underline{not} work under the completely-$\mathcal S$ hypothesis. More precisely, we will assume that:
\begin{assumption}
\label{as:quasi_comp-S-1}
The reflection matrix $R$ is not $\mathcal S$.  
\end{assumption}
\begin{assumption}
\label{as:quasi_comp-S-2}
All principal, strict sub-matrices of $R$ are completely-$\mathcal S$. 
\end{assumption}

Before going further, observe that the appearance of $\mathcal S$-matrices is very natural in the present context. Indeed, for instance, $R$ is an $\mathcal S$-matrix if and only if there exists a convex combination of reflection vectors which belongs to the interior of the orthant. Such a condition would allow us to define the process as a semimartingale after the time of hitting the origin. Similarly, the fact that a given principal sub-matrix of $R$ is $\mathcal S$ translates into the property that it is possible to define the process as a semimartingale after its visits on the corresponding face.

Therefore, as we shall prove, the probabilistic counterpart of Assumptions~\ref{as:quasi_comp-S-1} and \ref{as:quasi_comp-S-2} is that we can define the process $(Z_t)_{t\geq0}$ as a semimartingale before time
\begin{equation}
\label{eq:def_absorption_time}
   T:=\inf \{t>0: Z_t=0 \}\leq \infty,
\end{equation}
but not for $t\geq T$. For this reason, we will call $T$ in \eqref{eq:def_absorption_time} the absorption time:
 if the process hits the apex of the cone, then $T<\infty$ and we will say that the process is absorbed at the origin. Indeed, because of Assumption~\ref{as:quasi_comp-S-1}, there is no convex combination of reflection vectors belonging to the orthant, and consequently, we cannot define the process as a semimartingale after time $T$. However, our process is a semimartingale in the random time interval $[0,T]$; this will be proved in Proposition~\ref{prop:existence_ASRBM}.

We will also assume that:
\begin{assumption}
\label{as:drift>0}
The drift of the RBM is positive, that is, all coordinates of $\mu$ are positive.
\end{assumption}

Under Assumptions~\ref{as:quasi_comp-S-1}, \ref{as:quasi_comp-S-2} and \ref{as:drift>0}, our process exhibits the following dichotomy: either it hits the origin of the cone in finite time, i.e., $T<\infty$, or it goes to infinity (in the direction of the drift) before hitting the apex, i.e., $T=\infty$ and $\vert Z_t\vert \to\infty$ as $t\to\infty$. See Figure~\ref{fig:abs-esc}. We will prove this dichotomy in Proposition~\ref{prop:dichotomy}. 

This leads us to ask the following questions: what is the absorption probability 
\begin{equation}
\label{eq:absorption_probability}
   f(x)=\mathbb P_x[T<\infty]?
\end{equation} 
Equivalently, what is the escape probability
\begin{equation*}
   \mathbb P_x[T=\infty]=1-\mathbb P_x[T<\infty]=\mathbb{P} [\vert Z_t\vert\to \infty ]?
\end{equation*}
These questions are not only of theoretical nature: they also admit natural interpretations in population biology problems, in terms of extinction times of multitype populations \cite{LaRa-13}, or in risk theory, in terms of ruin of companies that collaborate to cover their mutual deficits \cite{AlAzMu-17,BaBoReWi-14,IvBo-15}.

\begin{figure}
\includegraphics[scale=0.7]{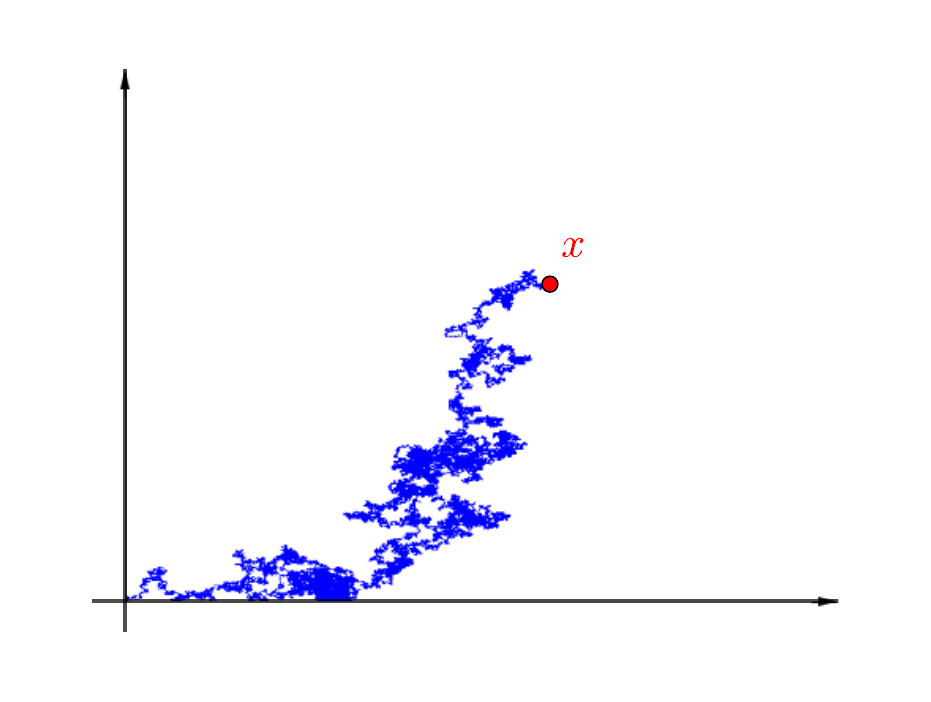} 
\hspace{1cm}
\includegraphics[scale=0.7]{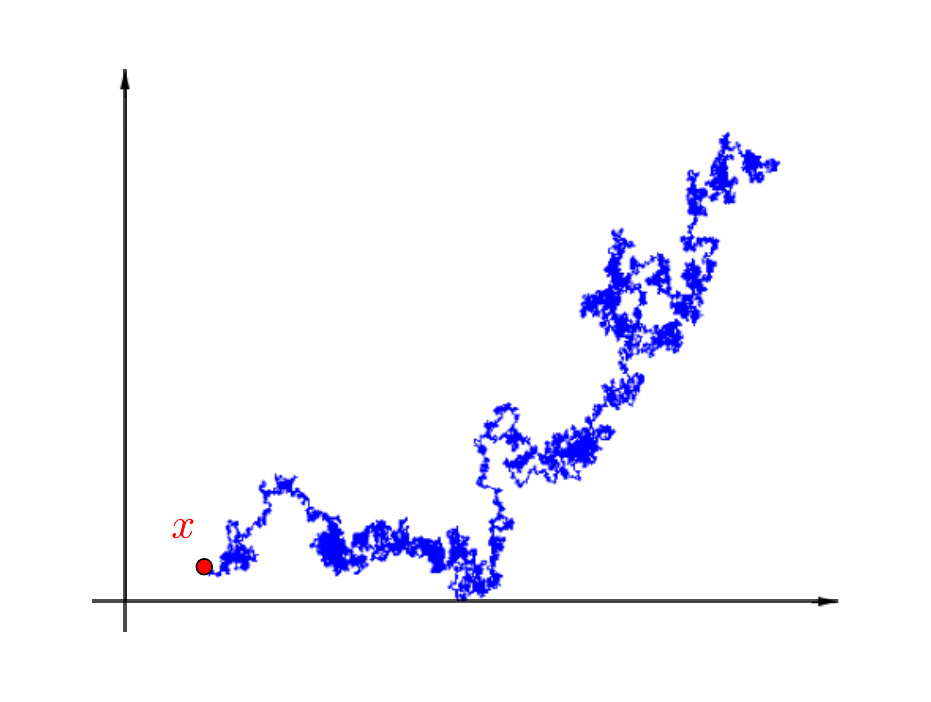} 
\caption{Two examples of paths of the process $(Z_t)_{t\geq0}$, with starting point $x$ marked in red. On the left, we have $T<\infty$, meaning that the process is absorbed in finite time at the apex of the cone. On the right, the process seems to escape to infinity, meaning that $T=\infty$.}
\label{fig:abs-esc}
\end{figure}

Because of its somehow dual nature, the problem of computing the absorption (or escape) probability is a priori as difficult as the problem of computing the stationary distribution in the semimartingale, recurrent case. Therefore, a natural question is to find an analogue of the skew symmetry \cite{Ha-78,HaWi-87} in this context, which we recalled here in \eqref{eq:parameters_product_form} and \eqref{eq:skew_symmetry}. The main result of the article is given in Theorem~\ref{thm:main} below. It is stated under four assumptions; while the first three have already been introduced, the final one, Assumption~\ref{as:neumann}, is of more technical nature and will be presented in Section~\ref{sec:PDE}. We conjecture that Assumption~\ref{as:neumann} is always true. For $x=(x_1, \ldots , x_n)\in \mathbb{R}_+^d$, $f(x)=\mathbb P_x[T<\infty]$ denotes the absorption probability~\eqref{eq:absorption_probability}.

\begin{theorem}[Dual skew symmetry in an orthant]
\label{thm:main}
Under Assumptions~\ref{as:quasi_comp-S-1}, \ref{as:quasi_comp-S-2}, \ref{as:drift>0} and \ref{as:neumann},
the following statements are equivalent:
\begin{enumerate}[label={\rm(\roman{*})},ref={\rm(\roman{*})}]
   \item\label{thm:main_it1}The absorption probability has a product form, i.e., there exist functions $f_1,\ldots ,f_d$ such that
\begin{equation*}
   f(x)=f_1(x_1)f_2(x_2) \cdots f_d(x_d).
\end{equation*}
\item\label{thm:main_it2}The absorption probability is exponential, i.e., there exists $a\in\mathbb{R}^d\setminus \{0\}$ such that
\begin{equation*}
   f(x)= \exp(a\cdot x).
\end{equation*}
\item\label{thm:main_it3}The reflection vectors $R_1,\ldots,R_d$ defined in \eqref{eq:def_reflection_matrix} and \eqref{eq:def_reflection_vector} are coplanar, that is,
\begin{equation*}
   \det R =0.
\end{equation*}
\end{enumerate}
When these properties are satisfied, the vector $a=(a_1,\ldots,a_n)$ in \ref{thm:main_it2} has negative coordinates and is the unique non-zero vector such that
\begin{equation}
\label{eq:a}
  a R =0\quad \text{and}\quad a  \Sigma \cdot a +a \mu=0.
\end{equation}
\end{theorem}
We refer to Figures~\ref{fig2:dim-2} and \ref{fig:dim-3} for a geometric illustration of the condition $\det R =0$ appearing in \ref{thm:main_it3}. See Figure~\ref{fig:ellipsoide} for a geometric illustration of the exponential decay rate $a$ in \eqref{eq:a}. When the parameters satisfy the assumptions (and conclusions) of Theorem~\ref{thm:main}, we will say that the model satisfies the \textit{dual skew symmetry} condition. This terminology will be explained in more detail in Remark~\ref{rem:dual}. In the case of dimension $2$, Theorem~\ref{thm:main} is proved in \cite{ErFrHu-20}. Assumption~\ref{as:neumann} will be discussed in Remark~\ref{rem:as4}. Note that the proof of \ref{thm:main_it3}$\Rightarrow$\ref{thm:main_it2}$\Rightarrow$\ref{thm:main_it1} does not use Assumption~\ref{as:neumann}.

\begin{figure}
\includegraphics[scale=0.6]{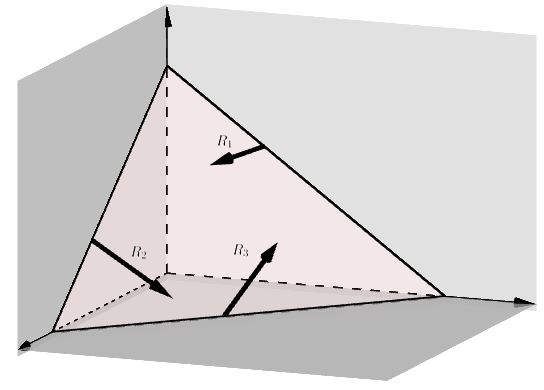} 
\caption{Condition $\det R =0$ in a $3$-dimensional orthant: the reflection vectors $R_1$, $R_2$ and $R_3$ are coplanar.}
\label{fig:dim-3}
\end{figure}

\begin{figure}
\includegraphics[clip=true,trim=0cm 1cm 0cm 0cm,scale=1.3]{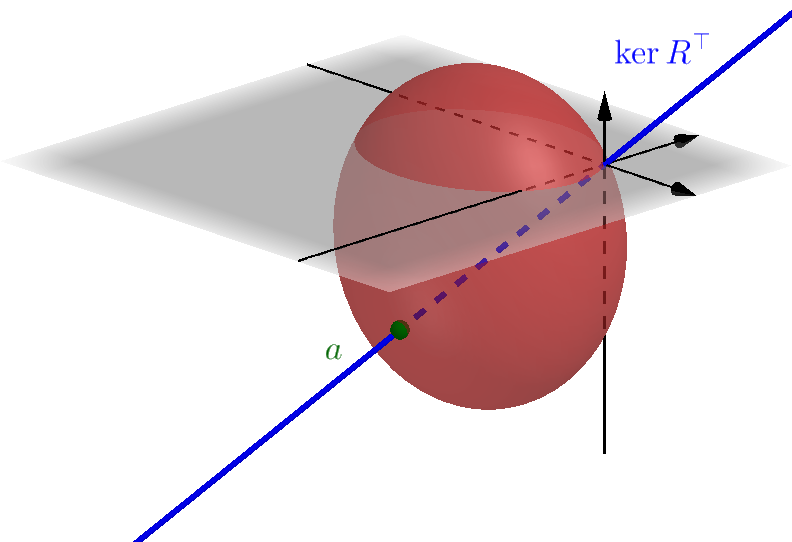} 
\caption{In red color: the ellipsoid with equation $x\Sigma \cdot x +\mu x =0$; in blue: $\ker R^\top$ (of dimension one by Lemma~\ref{lemme2}); in green: the exponential decay rate $a$.}
\label{fig:ellipsoide}
\end{figure}

\subsection*{Structure of the paper}
\begin{itemize}
   \item Section \ref{sec:absorption}: We define properly the process and show some of its pathwise properties. In particular, Proposition~\ref{prop:dichotomy} shows the dichotomy behavior (absorption vs.\ escape at infinity).
   \item Section \ref{sec:PDE}: We state and prove a PDE for the density of the absorption probability (Proposition~\ref{prop:PDE}). This PDE is dual to the one satisfied by the stationary distribution in the recurrent case.
   \item Section \ref{sec:skew}: We provide a proof of our main Theorem~\ref{thm:main}.
   \item Section \ref{sec:gen}: We propose a generalization of Theorem~\ref{thm:main} with absorption on facets, not necessarily the origin.
\end{itemize}

\section{Definition and first properties of the absorbed reflected Brownian motion}
\label{sec:absorption}

\subsection*{Existence and definition}

Let $(W_t)_{t\geq0}$ be a $d$-dimensional Brownian motion of covariance matrix $\Sigma$. Let $\mu \in\mathbb{R}^d$ be a drift, and let $R$ be a $d$-dimensional square matrix \eqref{eq:def_reflection_matrix} with coefficients $1$ on the diagonal.
\begin{proposition}[Existence of an absorbed SRBM]
\label{prop:existence_ASRBM}
Under Assumption \ref{as:quasi_comp-S-2}, there exists an absorbed SRBM in the orthant, i.e., a semimartingale defined up to the absorption time $T\leq\infty$ as in \eqref{eq:def_absorption_time} and such that for all $t \leqslant T$,
\begin{equation}
\label{eq:RBM_semimart}
   Z_t = x +W_t+\mu t + R L_t,
\end{equation}
where $L_t$ is a vector whose $i$th coordinate $L_t^i$ is a continuous, non-decreasing process starting from $0$, which increases only when the $i$th coordinate of the process $Z_t^i = 0$, and which is called the local time on the corresponding orthant face.
\end{proposition}
Under the additional hypothesis that $R$ is completely-$\mathcal S$, Proposition~\ref{prop:existence_ASRBM} is most classical: in this case, the RBM is well defined as the semimartingale \eqref{eq:RBM_semimart}, actually for any $t\in[0,\infty)$. Our contribution here is to prove that if $R$ is not an $\mathcal S$-matrix (our Assumption~\ref{as:quasi_comp-S-1}) and is therefore not completely-$\mathcal S$, then it is still possible to define the RBM as a semimartingale on the time interval $[0,T]$.

\begin{proof}
Although Proposition~\ref{prop:existence_ASRBM} is not formally proved in Taylor's PhD thesis \cite{Ta-90}, all necessary tools may be found there. More precisely, Taylor proves that when $R$ is completely-$\mathcal S$ (i.e., our Assumption~\ref{as:quasi_comp-S-2} plus the fact that $R$ is $\mathcal S$), then the RBM is an orthant $\mathbb R_+^d$ exists as a semimartingale globally on $[0,\infty)$. The proof in \cite{Ta-90} is then split into two parts:
\begin{itemize}
   \item First, \cite[Chap.~4]{Ta-90} shows that the SRBM exists on $[0,T]$, with $T$ defined in \eqref{eq:def_absorption_time}. The fact that $R$ is an $\mathcal S$-matrix is nowhere used in the proof: the only needed hypotheses are that all principal, strict sub-matrices are completely-$\mathcal S$ (our Assumption~\ref{as:quasi_comp-S-2}).
   \item As a second step, in \cite[Chap.~5]{Ta-90} (see in particular her Lemma~5.3), Taylor proves that if $R$ is an $\mathcal S$-matrix, then it is possible for the process started at the origin to escape the origin and to be well defined as a semimartingale.
\end{itemize}
Using only the first part of her arguments readily entails our Proposition~\ref{prop:existence_ASRBM}.
\end{proof}

\subsection*{Absorption and escape in asymptotic regimes}

We first prove two results which are intuitively clear, namely, that the absorption probability tends to one (resp.\ zero)\ when the starting point approaches the origin (resp.\ infinity), see Proposition~\ref{prop:absorption_origin} (resp.\ Proposition~\ref{prop:absorption_infinity}). Then, we will prove in Proposition~\ref{prop:dichotomy} the dichotomy already mentioned: either the process is absorbed in finite time, or it escapes to infinity as time goes to infinity. By convention, we will write $x\to 0$ (resp.\ $x\to\infty$) to mean that $\vert x\vert\to0$ (resp.\ $\vert x\vert\to\infty$) in the cone.

\begin{proposition}[Absorption starting near the origin]
\label{prop:absorption_origin}
One has
\begin{equation*}
   \lim_{x\to 0}\mathbb{P}_x [T<\infty] = 1.
\end{equation*}
\end{proposition}

\begin{proposition}[Absorption starting near infinity]
\label{prop:absorption_infinity}
One has
\begin{equation*}
   \lim_{x \to \infty}\mathbb{P}_x [T<\infty] = 0.
\end{equation*}
\end{proposition}
 
\begin{proposition}[Complementarity of escape and absorption]
\label{prop:dichotomy}
When $T=\infty$, then almost surely $\lim_{t\to\infty} \vert Z_t\vert = \infty$, i.e.,
\begin{equation*}
  \mathbb{P}_x \left[ \left. \lim_{t\to\infty} \vert Z_t\vert = \infty \right\vert T=\infty \right]   = 1.
\end{equation*}
This implies that $\mathbb{P}_x[T=\infty]=\mathbb{P}_x[\vert Z_{t\wedge T}\vert \to \infty ]$.
\end{proposition}

\begin{proof}[Proof of Proposition~\ref{prop:absorption_origin}]
Let us define $\tau_x = \inf \{t>0 : x+W_t+\mu t <0  \}$ (by convention $\inf \emptyset =\infty$), and consider the set
\begin{equation*}
   \{ \tau_x <\infty \} = \{ \exists t>0 \text{ such that } x+W_t+\mu t <0 \}.
\end{equation*} 
The proof consists in two steps. We first prove that $\{ \tau_x <\infty \} \subset \{ T<\infty \}$ and then show that $\lim_{x\to 0}\mathbb{P}[ \tau_x <\infty ] =1$.
\begin{enumerate}[label={\rm Step~\arabic{*}.},ref={\rm Step~\arabic{*}}]
   \item\label{it:step1}
Assume that $\tau_x < \infty$ and fix a $t<\infty$ such that
$x+W_t+\mu t<0$. We are going to show that $T\leqslant t$. 
We proceed by contradiction and assume that $t<T$. Then from \eqref{eq:RBM_semimart} we get that
\begin{equation*}
   R L_t =  Z_t - x-W_t-\mu t >0.
\end{equation*}
The last inequality comes from the fact that $Z_t\geqslant 0$ and that $x+W_t+\mu t <0$.
Remembering that $L_t\geqslant 0$, the fact that $ R L_t>0$ implies that $R$ is an $\mathcal{S}$-matrix, which contradicts Assumption~\ref{as:quasi_comp-S-1}.
We conclude that $T\leqslant t<\infty$. We have thus shown that $\{ \tau_x <\infty \} \subset \{ T<\infty \}$.
\item\label{it:step2} By Blumenthal's zero--one law, we have
\begin{equation*}
    \mathbb{P}[ \tau_0=0 ] =1, 
\end{equation*}
since
\begin{equation*}
   \{\tau_0 =0\}=\bigcap_{t>0} \left\{ \inf_{s \leqslant t}  (W_s+\mu s) <0 \right\}  \in \mathcal{F}_{0+}= \bigcap_{t>0} \mathcal{F}_t,
\end{equation*}
where $\mathcal{F}_t =\sigma \{ W_s , s\leqslant t \}$.
This implies that $\mathbb{P}[\tau_0 <\infty ]=1$. We deduce that almost surely, there exists $t_0$ such that $W_{t_0}+\mu t_0<0$, and then for all $x<-W_{t_0}-\mu t_0$ we have $\tau_x <\infty$. Then 
$\mathds{1}_{\{\tau_x < \infty\}}\underset{x\to 0}{\longrightarrow} 1 \text{ a.s.}$, and by dominated convergence we have 
\begin{equation*}
   \mathbb{P}[ \tau_x <\infty ]=\mathbb{E}[\mathds{1}_{\{\tau_x < \infty\}}] \underset{x\to 0}{\longrightarrow} 1.
\end{equation*} 
\end{enumerate}
Thanks to \ref{it:step1} and \ref{it:step2}, we conclude that $\mathbb{P}_x[T<\infty]\geqslant \mathbb{P}[ \tau_x <\infty ]$ and therefore $\underset{x\to 0}{\lim}\mathbb{P}_x [T<\infty] = 1$, using the above estimate.
\end{proof}
\begin{proof}[Proof of Proposition~\ref{prop:absorption_infinity}]
Introduce the event
\begin{equation*}
   B_x = \{\forall t\in\mathbb{R},\, x+W_t+\mu t >0\}.
\end{equation*}
For any element belonging to $B_x$, then 
$Z_t=x+W_t+\mu t$ for all $t\in \mathbb{R}$ (the process never touches the boundary of the orthant, meaning no reflection on the boundary). We deduce that $Z_t>0$ for all $t\in \mathbb{R}$ and then that $B_x \subset \{ T=\infty \}$. Therefore,
\begin{equation*}
   \mathbb P[B_x]\leqslant \mathbb{P}_x [T=\infty].
\end{equation*}
To conclude, we are going to show that $\lim_{x\to \infty} \mathbb{P}[B_x]=1$. It comes from the fact that a.s.\ $\inf_{t>0} \{ W_t +\mu t \}> -\infty$, since $\mu>0$ by Assumption~\ref{as:drift>0}. For all $x> -\inf_{t>0} \{ W_t +\mu t \}$, we have $x+W_t+\mu t >0$ for all $t$. We deduce that
$
\mathds{1}_{\{\forall t\in\mathbb{R}, x+W_t+\mu t >0\}}\underset{x\to \infty}{\longrightarrow} 1 \text{ a.s.}
$,
and by dominated convergence, we have 
\begin{equation*}
   \mathbb{P}[ B_x ]=\mathbb{E}[\mathds{1}_{\{\forall t\in\mathbb{R}, x+W_t+\mu t >0\}}] \underset{x\to \infty}{\longrightarrow} 1.\qedhere
\end{equation*}
\end{proof}

Before proving Proposition~\ref{prop:dichotomy}, we first recall some useful definitions and properties related to recurrence and transience of Markov processes. All of them are most classical, but having them here stated clearly will facilitate our argument.
These results and their proofs may be found in~\cite{Az-66}.

Consider a continuous, strong Feller Markov process $X_t$ on a locally
compact state space $E$ with countable basis. For $V\subset E$, let us define $\tau_V = \inf \{t > 0 : X_t \in V \}$. 
\begin{itemize}
   \item The point $x\in E$ is said to be recurrent if for all neigbourhoods $V$ of $x$,
\begin{equation*}
   \mathbb{P}[\limsup \mathds{1}_V (X_t)=1]=1.
\end{equation*}
   \item If a point is not recurrent, it is said to be transient. In this case, by \cite[Thm.~III~1]{Az-66}, there exists a neigbourhood $V$ of $x$ such that $\mathbb{P}[\limsup \mathds{1}_V (X_t)=1]=0$.
\item The point $x$ is said to lead to $y$ if for all neighbourhoods $V$ of $y$, we have $\mathbb{P}_x [\tau_V < \infty] > 0$.
The points $x$ and $y$ are said to communicate if $x$ leads to $y$ and $y$ leads to $x$. This defines an equivalence relation. 
   \item If two states communicate, they are both transient or both recurrent \cite[Prop.~IV~2]{Az-66}.
   \item If all points are transient, then $X_t$ tends to $\infty$ as $t\to \infty$ almost surely \cite[Prop.~III~1]{Az-66}.
\end{itemize}

\begin{proof}[Proof of Proposition~\ref{prop:dichotomy}]
Define the process $(\widetilde{Z}_t)_{t\geq 0}$ as the process $(Z_t)_{t\geq 0}$ conditioned never to hit $0$ in a finite time. The transition semigroup of this new Markov process $\widetilde{Z}_t$ is defined, for $x\in \mathbb{R}_+^d \setminus \{ 0\}$ and $V \subset \mathbb{R}_+^d \setminus \{ 0\}$, by 
\begin{equation*}
   \mathbb{P}_x[\widetilde{Z}_t \in V]=\mathbb{P}_x[{Z}_t \in V \vert T=\infty]. 
\end{equation*}
All points of $\mathbb{R}_+^d \setminus \{0\}$ communicate, they thus constitute a unique equivalence class. We deduce that they all are transient or all recurrent. It is thus enough to show that one of them is transient, to show that they all are. 

Let us take a point in the interior of $\mathbb{R}_+^d \setminus \{0\}$, for example $x=(1,\ldots, 1)$. Since $\mu >0$, by standard properties of Brownian motion we have
\begin{equation*}
   \mathbb{P}[\forall t \in\mathbb{R},\,x+W_t+\mu t>0]>0.
\end{equation*}
In dimension one, this property directly derives from \cite[Eq.~1.2.4(1)]{BoSa-12} (on p.~252); it easily generalizes to all dimensions.
When this event of positive probability occurs, the process never touches the boundary and thus $\widetilde{Z}_t= x+W_t+\mu t \to \infty$ and $\limsup \mathds{1}_V (\widetilde{Z}_t)=0$, for any $V$ relatively compact neighbourhood of $x$. We have shown that there exists a neighbourhood $V$ of $x$ such that $\mathbb{P}_x [\limsup \mathds{1}_V (\widetilde{Z}_t)=1]<1$, which implies that $x$ is not recurrent and is then transient.

Using \cite[Prop.~III~1]{Az-66} allows us to conclude, since as recalled above, if all points are transient, then the process tends to infinity almost surely.
\end{proof}

\section{Partial differential equation for the absorption probability}
\label{sec:PDE}

In a classical way, the generator of the Brownian motion in the interior of the orthant is defined by 
\begin{equation*}
   \mathcal{G}f (x)= \lim_{t\to 0} \frac{\mathbb{E}_x[f(Z_t)] - f(x) }{t}= \frac{1}{2}(\nabla \cdot \Sigma \nabla f)  (x) + (\mu \cdot \nabla f)  (x),
\end{equation*}
where we assume that $f$ is bounded in the first equality and that $f$ is twice differentiable in the second equality. In the rest of the paper, the following assumption is made.

\begin{assumption}
\label{as:neumann}
For all continuous, bounded functions $g$, the transition semigroup
\begin{equation*}
   x\mapsto P_t g(x) := \mathbb{E}_x [g(Z_{t\wedge T})]
\end{equation*}
is differentiable, and satisfies the Neumann boundary condition $R_i \cdot P_t g(x)=0$ on the $i$th face of the orthant $x_i=0$.
\end{assumption}
\begin{remark}[Plausibility of Assumption~\ref{as:neumann}]
\label{rem:as4}
Many evidences suggest that this hypothesis is true:
\begin{itemize}
   \item By \cite[Cor.~3.3]{An-09}, Assumption~\ref{as:neumann} is true provided we replace $T$ by the first hitting time of the intersection of two faces, or assuming that the process does not hit the intersection of two faces. 
   \item As a consequence of the above, Assumption~\ref{as:neumann} is true in dimension two. 
   \item By \cite{DeZa-05}, Assumption~\ref{as:neumann} holds true in the particular case of orthogonal reflections.
   \item Assumption~\ref{as:neumann} is stated as a conjecture in \cite[(8.2b)]{HaRe-81a}; however, the latter article does not attempt to prove rigorously these regularity questions.
   \item The paper \cite{LiRa-19} shows in full generality the pathwise differentiability with respect to the starting point $x$. We believe that a way to attack the proof of Assumption~\ref{as:neumann} could be to combine the results of \cite{LiRa-19} with the computations made in the proof of \cite[Cor.~3.3]{An-09}.
\end{itemize}
\end{remark}

\begin{proposition}[Partial differential equation]
\label{prop:PDE}
Under Assumptions~\ref{as:quasi_comp-S-1}, \ref{as:quasi_comp-S-2}, \ref{as:drift>0} and \ref{as:neumann},
the absorption probability \eqref{eq:absorption_probability} is the unique function $f$ which is
\begin{itemize}
   \item bounded and continuous in the interior of the orthant $\mathbb R_+^d$ and on its boundary,
   \item continuously differentiable in the interior of the orthant and on its boundary (except perhaps at the corner),
\end{itemize}
and which further satisfies the PDE:
\begin{itemize}
   \item $\mathcal{G} f =0$ on the orthant (harmonicity),
   \item $R_i \cdot \nabla f = 0$ on the $i$th face of the orthant $x_i=0$ (Neumann boundary condition),
   \item $f(0)=1$ and $\lim_{x\to \infty} f(x)=0$ (limit values).
\end{itemize}
\end{proposition}

\begin{proof}
The proof is similar to \cite[Prop.~11]{ErFrHu-20}. We start with the \textit{sufficient condition}. Dynkin's formula leads to
\begin{equation*}
\mathbb{E}_x[f(Z_{t\wedge T})] = f(x) +\mathbb{E}_x \int_0^{t\wedge T} \mathcal{G}f (Z_s) \mathrm{d}s + 
\sum_{i=1}^d \mathbb{E}_x \int_0^{t\wedge T} 
(R \nabla f )_i
\mathrm{d}L^i_s.
\end{equation*}
There is a technical subtlety in applying Dynkin's formula, as the latter requires functions having a $\mathcal C^2$-regularity, which is a priori not satisfied at the origin in our setting. However, we may first apply this formula for $T_n=\inf\{t>0 : \vert Z_t\vert <\frac{1}{n}\}<T$, then $f$ has the desired regularity on $\mathbb R_+^d\setminus \{x:\vert x\vert<\frac{1}{n}\}$. One may conclude as $T_n$ converges increasingly to $T$ as $n\to\infty$.

Since $f$ is assumed to satisfy the PDE stated in Proposition~\ref{prop:PDE}, the latter sum of three terms is simply equal to $f(x)$. We further compute
\begin{equation*}
   f(x) = \mathbb{E}_x[f(Z_{t\wedge T})] = 
\mathbb{E}_x[f(Z_{t\wedge T})\mathds{1}_{T\leqslant t}]
+
\mathbb{E}_x[f(Z_{t\wedge T})\mathds{1}_{T> t}]
=f(Z_T) \mathbb P_x[T\leqslant t]+
\mathbb{E}_x[f(Z_{t\wedge T})\mathds{1}_{T> t}].
\end{equation*}
As $t\to\infty$, the above quantity converges to
\begin{equation*}
   f(0)\mathbb P_x[T<\infty] + 
\mathbb{E}_x\left[\lim_{t\to\infty}f(Z_{t})\mathds{1}_{T=\infty}\right]
= \mathbb P_x[T<\infty],
\end{equation*}
where the last equality comes from the limit values $\lim_{x\to \infty} f(x)=0$, $f(0)=1$ and from Proposition~\ref{prop:dichotomy}, which together imply that when $T=\infty$ we have $\lim_{t\to\infty}Z_{t}=\infty$.
We immediately deduce that
$f(x)=\mathbb P_x[T<\infty]$.

We now move to the \textit{necessary condition}. We denote the absorption probability~\eqref{eq:absorption_probability} by $f$ and show that it satisfies the PDE of Proposition~\ref{prop:PDE}. Consider the event $\{ T<\infty \} \in \mathcal{F}_\infty$ and define 
\begin{equation*}
   M_t=\mathbb{E} [\mathds{1}_{\{ T<\infty \}} \vert \mathcal{F}_{t \wedge T}],
\end{equation*}
which is a $\mathcal{F}_t$-martingale. Observe that $M_0=f(x)$ and, by the Markov property, $M_t=f(Z_{t \wedge T})$. We deduce that $\mathbb{E}_x[f(Z_{t \wedge T})]=\mathbb{E}[M_t \vert \mathcal{F}_0]=M_0=f(x)$.
By definition of $\mathcal{G}$, we obtain that for $x$ in the interior of the orthant,
\begin{equation*}
   \mathcal{G}f (x)= \lim_{t\to 0} \frac{\mathbb{E}_x[f(Z_t)] - f(x) }{t} = 0.
\end{equation*}
The Neumann boundary condition and the differentiability follow from the fact that $f(x)=\mathbb{E}_x[f(Z_{t \wedge T})]$ and from Assumption~\ref{as:neumann}.
The limit values follow from our Propositions~\ref{prop:absorption_origin} and \ref{prop:absorption_infinity}.
\end{proof}

\begin{remark}[Duality between absorption probability and stationary distribution]
\label{rem:dual}
Let us define the dual generator $
\mathcal{G}^* f (x)= \frac{1}{2}(\nabla \cdot \Sigma \nabla f)  (x) - (\mu \cdot \nabla f)  (x)
$ 
as well as the matrix $R^* = 2 \Sigma -R \ \textnormal{diag} (\Sigma )$, whose columns are denoted by $R^*_i$. In the recurrent case, the stationary distribution satisfies the following PDE, see \cite[Eq.~(8.5)]{HaRe-81a}:
\begin{itemize}
\item $\mathcal{G}^* f =0$ in the orthant, 
\item $R^*_i \cdot \nabla f -2 \mu_i f= 0$ on the $i$th face of the orthant defined by $x_i=0$. 
\end{itemize}
As a consequence, the absorption probability satisfies a PDE (Proposition~\ref{prop:PDE}), which is dual to the one which holds for the stationary distribution.
\end{remark}

\section{Dual skew symmetry: proof of the main result}
\label{sec:skew}

This section is devoted to the proof of Theorem~\ref{thm:main}, which establishes the dual skew symmetry condition. We first prove two technical lemmas on the reflection matrix $R$ in \eqref{eq:def_reflection_matrix}.
\begin{lemma}
\label{lemme}
If $R$ satisfies Assumptions~\ref{as:quasi_comp-S-1} and \ref{as:quasi_comp-S-2}, then for all $i$, there exists $j\neq i$ such that $r_{ij} \neq 0$.
\end{lemma}
\begin{proof}
It is enough to prove Lemma~\ref{lemme} for $i=1$, as we would show the other cases similarly.
Consider $\widetilde R$ the principal submatrix of $R$ obtained by removing the first line and the first column. This matrix is completely $\mathcal{S}$ by Assumption~\ref{as:quasi_comp-S-2}, so that there exists $\widetilde X=(x_2,\ldots,x_d)^\top\geqslant 0$
 such that $\widetilde R \widetilde X >0$. Consider now $\widetilde C_1=(r_{21},\ldots,r_{d1})^\top$, which is the first column of $R$ without its first coordinate. Let us choose $\lambda>0$ large enough such that $ \widetilde C_1 + \lambda \widetilde R \widetilde X >0$. If for all $j\neq 1$ we have $r_{1j}=0$, then for 
  $
  X=(1,\lambda x_2,\ldots,\lambda  x_d)^\top$
  we would have 
\begin{equation*}
    RX
 =
 \begin{pmatrix} 1 & 0  \cdots  0 \\ 
  \widetilde C_1 &  \widetilde R  \\
 \end{pmatrix}
  \begin{pmatrix}
   1  \\
   x_2 \\
    \vdots  \\
    x_d
  \end{pmatrix}  
  =\begin{pmatrix} 1 \\ \widetilde C_1+\lambda \widetilde R \widetilde X  \end{pmatrix}>0
\end{equation*} 
and then $R$ would be an $\mathcal{S}$-matrix, contradicting our Assumption~\ref{as:quasi_comp-S-1}. 
\end{proof}

\begin{lemma}
\label{lemme2}
If $R$ satisfies Assumptions~\ref{as:quasi_comp-S-1} and \ref{as:quasi_comp-S-2}, and if in addition $\det R=0$, then $R$ has rank $d-1$, and there exist a positive column vector $U>0$ in $\ker R$ and a positive row vector $a>0$ such that $aR=0$.
\end{lemma}

\begin{proof}
The rank of the matrix $R$ is obviously $\leqslant d-1$, since $\det R=0$. We now show that the rank is $\geqslant d-1$. 

Let $\widetilde R_j$ be the submatrix of $R$ obtained by removing the $j$th line and the $j$th column. These matrices are $\mathcal{S}$-matrices by Assumption~\ref{as:quasi_comp-S-2}, and we can choose 
\begin{equation*}
   \widetilde X_j=\begin{pmatrix}
  \widetilde x_{1j}  \\
    \vdots  \\
 \widetilde   x_{(j-1)j}  \\
 \widetilde   x_{(j+1)j}  \\
    \vdots  \\
  \widetilde  x_{dj}
  \end{pmatrix} \geqslant 0
  \quad \text{such that}
  \quad 
  \widetilde R_j \widetilde X_j
=\widetilde Y_j=\begin{pmatrix}
   \widetilde y_{1j}  \\
    \vdots  \\
    \widetilde y_{(j-1)j}  \\
    \widetilde y_{(j+1)j}  \\
    \vdots  \\
    \widetilde y_{dj}
  \end{pmatrix}  
   >0
  .
\end{equation*} 
We now define the vertical vectors $X_j^\varepsilon=(\widetilde x_{ij})_{i=1,\ldots, d}$, setting $\widetilde x_{jj}=\varepsilon$ for some $\varepsilon>0$. We have 
\begin{equation*}
   R X_j^\varepsilon=Y_j^\varepsilon
 =(y_{ij}^\varepsilon)_{i=1,\ldots, d},
\end{equation*}  
where $y_{ij}^\varepsilon= \varepsilon r_{ij}+ \widetilde y_{ij}>0$ for $i \neq j$ and $\varepsilon>0$ small enough, and $y_{jj}^\varepsilon=\varepsilon + \widetilde L_j \widetilde X_j$, where we set 
\begin{equation*}
   \widetilde L_j=(r_{1j}, \ldots, r_{(j-1)j},r_{(j+1)j},\ldots ,r_{dj})
\end{equation*}
the $j$th line of $R$, with the $j$th coordinate $r_{jj}=1$ excluded. Since $R$ is not an $\mathcal{S}$-matrix by our Assumption~\ref{as:quasi_comp-S-1}, we must have $y_{jj}^\varepsilon=\varepsilon+ \widetilde L_j \widetilde X_j \leqslant 0$. We deduce that $y_{jj}^0= \widetilde L_j \widetilde X_j \leqslant -\varepsilon <0$.

Then, introducing the vectors 
\begin{equation*}
   X_j=\frac{1}{-y_{jj}^0 } X_j^0 \geqslant 0
\quad \text{and}
  \quad
  Y_j=(y_{ij})_{i=1,\ldots ,d}  =R X_j,
\end{equation*}
we have 
\begin{equation*}
   Y_j=R X_j=\begin{pmatrix}
   y_{1j}  \\
    \vdots  \\
    y_{(j-1)j}  \\
   - 1 \\
    y_{(j+1)j}  \\
    \vdots  \\
    y_{dj}
  \end{pmatrix} ,
  \quad\text{where } y_{ij}=\frac{\widetilde y_{ij}}{-y_{jj}^0 } >0 \text{ for } i\neq j.
\end{equation*} 
Denoting the matrix $P=(X_1, \ldots, X_d ) \geqslant 0$, we have 
\begin{equation*}
   -RP= \begin{pmatrix}
    1 & -y_{12} & \dots & -y_{1d} \\
    -y_{21} & 1 & \dots & -y_{2d} \\
    \vdots & \vdots & \ddots & \vdots \\
   - y_{d1} & -y_{d2} & \dots & 1
  \end{pmatrix}= 2 \textnormal{Id} -T,
  \text{ where } T=\begin{pmatrix}
    1 & y_{12} & \dots & y_{1d} \\
    y_{21} & 1 & \dots & y_{2d} \\
    \vdots & \vdots & \ddots & \vdots \\
   y_{d1} & y_{d2} & \dots & 1
  \end{pmatrix}>0.
\end{equation*}
All coefficients of $T$ are positive. Consequently, using Perron-Frobenius theorem, $T$ has a unique maximal eigenvalue $r$, its associated eigenspace is one-dimensional and there exists a positive eigenvector $V$ associated to $r$. Let us remark that  since $\det R=0$, then $\det (2 \textnormal{Id} -T)=\det(-RP)=0$ and $2$ is an eigenvalue of $T$. Then $r\geqslant 2$, and there are two cases to treat.
\begin{itemize} 
   \item Assume first that the maximal eigenvalue is $r>2$. Let $V>0$ be a positive associated eigenvector such that $TV=rV$. We deduce that $-RPV=2V-TV=(2-r)V$ and then $R(PV)=(r-2)V>0$, where $PV \geqslant 0$ since $P \geqslant 0$ and $ V > 0$. Then we have shown that $R$ is an $\mathcal{S}$-matrix, which contradicts Assumption~\ref{as:quasi_comp-S-1}. So we must be in the situation where $r=2$.
   \item If $r=2$ is the maximal eigenvalue of $T$, and $V>0$ the positive eigenvector such that $TV=2V$, then we have $RU=0$ for $U=PV>0$. Furthermore $\dim \ker(2\textnormal{Id}-T)=1$ and then $d-1=\text{rank} (2\textnormal{Id}-T)=\text{rank} RP \leqslant \text{rank} R$ and then $R$ has rank $d-1$.
\end{itemize}
Left eigenspaces of $T$ are (right) eigenspaces of $T^\top$. If we take $a$ such that $aR=0$, then $a$ belongs to the left eigenspace associated to the eigenvalue $2$ of $T$. By Perron-Frobenius theorem, we deduce that we can choose $a>0$.
\end{proof}

We now prove a result showing that the hitting probability of the origin is never $0$, for all starting points.
\begin{lemma}
For all $x\in \mathbb{R}_+^d$, $f(x) > 0$.
\label{lemma:fpositive}
\end{lemma}
\begin{proof}
By Proposition~\ref{prop:absorption_origin}, there exists a point $y_0$ in the interior of the orthant such that $f(y_0)> 0$. By continuity of $f$ (Proposition~\ref{prop:PDE}), we can find an open neighbourhood $U$ of $y_0$ such that $f(y) > 0$ for all $y \in U$. Then we conclude that
\begin{equation*}
   f(x)=\mathbb{E}_x [f(Z_{t\wedge T})]= \int_{\mathbb{R}_+^d}  f(y) \mathbb{P}_x(X_{t\wedge T} = \mathrm{d}y) \geqslant \int_{U}  f(y) \mathbb{P}_x(X_{t\wedge T} = \mathrm{d}y) >0.
\end{equation*}
(The first equality in the previous equation has already been proved in the proof of Proposition~\ref{prop:PDE}). 
\end{proof}

Let us now prove the main result.
\begin{proof}[Proof of Theorem~\ref{thm:main}]
\ref{thm:main_it1} $\Rightarrow$ \ref{thm:main_it2}:
We assume that $f(x)=f_1(x_1)\cdots f_d(x_d)$ and we denote $\partial \ln f_i = f'_i /f_i$ (note that due to Proposition~\ref{prop:PDE}, the functions $f$ and $f_i$ are differentiable and by Lemma~\ref{lemma:fpositive}, $f_i(x_i) \neq 0$ for all $i$ and all $x_i$). On the boundary $x_i=0$, the Neumann boundary condition of Proposition~\ref{prop:PDE} implies that
\begin{equation*}
   0=\frac{R_i \cdot \nabla f}{f}=R_i \cdot 
\begin{pmatrix}
    \partial \ln f_1 (x_1)  \\
    \vdots  \\
    \partial \ln f_d (x_d)
  \end{pmatrix}
  \text{ for } x_i=0.
\end{equation*}
In particular, for all $j \neq i$, taking $x_{i'}=0$ for all $i'\neq j$, we obtain
\begin{equation*}
   R_i \cdot 
\begin{pmatrix}
    \partial \ln f_1 (0)  \\
    \vdots  \\
    \partial \ln f_j (x_j)  \\
    \vdots  \\
    \partial \ln f_d (0)
  \end{pmatrix}
  =0.
\end{equation*}
We deduce that for all $i$ and $j$ such that $i \neq j$, the function 
$r_{ij} \partial \ln f_j (x_j) $ 
is a constant, which we can compute as $- \sum_{j'\neq j} r_{ij'} \partial \ln f_{j'} (0)$. By Lemma~\ref{lemme}, for all $j$ there exists $i \neq j$ such that $r_{ij}\neq 0$.
This implies that $\partial \ln f_j (x_j) $ is constant and then that $f_j$ is exponential: there exists $a_j$ such that $f_j(x_j)=e^{a_j x_j}$. The limit value $\lim_{x\to\infty} f(x) =0$ implies that $a \neq 0$.

\ref{thm:main_it2} $\Rightarrow$ \ref{thm:main_it1}: This implication is trivial by taking $f_i(x_i)=e^{a_i x_i}$.

\ref{thm:main_it2} $\Rightarrow$ \ref{thm:main_it3}: If $f(x)=e^{ax}$ satisfies the PDE of Proposition~\ref{prop:PDE}, then $R_i \cdot \nabla f(x)= a R_i  e^{ax}=0 $ on the boundary face $x_i=0$. We obtain that $ a R_i =0$ for all $i$ and then that $a R =0$. We deduce that $\det R=0$ since $a\neq 0$.

\ref{thm:main_it3} $\Rightarrow$ \ref{thm:main_it2}: If $\det R =0$, then by Lemma~\ref{lemme2} one has $\dim \ker R=1$, and we can choose $a' \in\mathbb{R}^d$ such that $a'>0$ and $ a' R =0$. Then $a=- \frac{a'\mu}{a'  \Sigma \cdot a'}\ a'<0$ is the unique vector which satisfies $a\cdot R=0$ and $a  \Sigma \cdot a +a \mu =0 $. Then it is easy to verify that $e^{ax}$ satisfies the PDE of Proposition~\ref{prop:PDE}, while the boundary condition at infinity comes from the fact that $a<0$.
\end{proof}


\section{A generalization of Theorem~\ref{thm:main}: absorption on a facet}
\label{sec:gen}

Theorem~\ref{thm:main} can be generalized to the case where the RBM is absorbed at a facet of the orthant, with equation 
\begin{equation*}
   x_{i_1}=\cdots=x_{i_k}=0,
\end{equation*}  
for some fixed $k\in \{1,\ldots, n \}$. The situation where $k=n$ is the case of an absorption at the apex of the cone, which is treated in detail in the present article. For the sake of brevity and to avoid too much technicality, we will not prove this generalization in this article, even though
all intermediate steps in the proof may be extended.

In the general case of a facet, let us state three assumptions which generalize Assumptions~\ref{as:quasi_comp-S-1},~\ref{as:quasi_comp-S-2} and~\ref{as:drift>0}. Let us define $\widetilde R$ (resp.\ $\widetilde \Sigma$) the principal sub-matrix of $R$ (resp.\ $\Sigma$), where we keep only the $i_1$th up to $i_k$th lines and columns.
\begin{itemize}
   \item The new Assumption~\ref{as:quasi_comp-S-1} is that the reflection matrix $\widetilde R$ is not $\mathcal{S}$. 
   \item The new second assumption is that all principal sub-matrices of $R$ which do not contain $\widetilde R$ are completely-$\mathcal{S}$. 
   \item The third assumption about the positivity of the drift $\mu>0$ remains unchanged (even though we could probably weaken this hypothesis).
\end{itemize}
Under these assumptions, we may define the reflected Brownian motion $(Z_t)_{t\geq 0}$ until time 
\begin{equation*}
   \widetilde T = \inf \{t>0 : Z_t^{i_1}=\cdots =Z_t^{i_k}=0 \},
\end{equation*} 
where $Z^i$ stands for the $i$th coordinate of $Z$. Let us denote the absorption probability
\begin{equation*}
   \widetilde f(x)=\mathbb P_x[\widetilde T<\infty].
\end{equation*} 
Then Theorem~\ref{thm:main} may be extended as follows. The following assertions are equivalent:
\begin{enumerate}[label={\rm(\roman{*}')},ref={\rm(\roman{*}')}]
   \item $f$ has a product form.
   \item $f$ is exponential, i.e., $f(x)= \exp (a_{i_1}x_{i_1}+ \cdots +a_{i_k}x_{i_k} )$ with $a_{i_j}\neq 0$.
   \item $\det \widetilde R =0$.
\end{enumerate} 
In this case, the vector $\widetilde a=(a_{i_1}, \ldots ,a_{i_k} )$ is negative and is the unique non-zero vector such that $\widetilde a \widetilde R =0$ and $\widetilde a  \widetilde\Sigma \cdot \widetilde a + \widetilde a \widetilde\mu=0$, where we defined the vertical vector $\widetilde \mu =({\mu_{i_j}})_{j=1,\ldots , k}$.

\subsection*{Acknowledgments}

We are grateful to Philip Ernst and to John Michael Harrison for very interesting discussions about topics related to this article. We thank an anonymous referee for very useful remarks and suggestions.


\small
\begin{thebibliography}{10}
\bibliographystyle{plain}

\bibitem{AlAzMu-17}
H. Albrecher, P. Azcue and N. Muler (2017). 
Optimal dividend strategies for two collaborating insurance companies.
\textit{Adv. Appl. Probab.} \textbf{49} 515--548

\bibitem{An-09}
S. Andres (2009).
Pathwise differentiability for SDEs in a convex polyhedron with oblique reflection.
\textit{Ann. Inst. H. Poincar\'e Probab. Statist.} \textbf{45} 104--116

\bibitem{Az-66}
J. Azéma, M. Kaplan-Duflo and D. Revuz (1966). 
Récurrence fine des processus de Markov. 
\textit{Ann. Inst. H. Poincar\'e Probab. Statist.} \textbf{2} 185--220 

\bibitem{BaFa-87}
F. Baccelli and G. Fayolle (1987).
Analysis of models reducible to a class of diffusion processes in the positive quarter plane.
\textit{SIAM J. Appl. Math.} \textbf{47} 1367--1385

\bibitem{BaBoReWi-14}
E. S. Badila, O. J. Boxma, J. A. C. Resing and E. M. M. Winands (2014). 
Queues and risk models with simultaneous arrivals.
\textit{Adv. Appl. Probab.} \textbf{46} (3) 812--831

\bibitem{BrDaHa-10}
M. Bramson, J. Dai and J. M. Harrison (2010).
Positive recurrence of reflecting Brownian motion in three dimensions.
\textit{Ann. Appl. Probab.} \textbf{20} 753--783

\bibitem{Br-11}
M. Bramson (2011).
Positive recurrence for reflecting Brownian motion in higher dimensions.
\textit{Queueing Syst.} \textbf{69} 203--215

\bibitem{BMEPFrHaRa-21}
M. Bousquet-M\'elou, A. Elvey Price, S. Franceschi, C. Hardouin and K. Raschel (2021).
The stationary distribution of reflected Brownian motion in a wedge: differential properties. 
\textit{Preprint} arXiv:2101.01562

\bibitem{BoSa-12}
A. N. Borodin and P. Salminen (2012).
\textit{Handbook of Brownian motion: Facts and formulae.}
2nd ed. 
Probability and Its Applications. 
{Birkh{\"a}user}

\bibitem{Ch-96}
H. Chen (1996).
A sufficient condition for the positive recurrence of a semimartingale reflecting Brownian motion in an orthant.
\textit{Ann. Appl. Probab.} \textbf{6} 758--765

\bibitem{DaHa-92}
J. Dai and J. M. Harrison (1992).
Reflected Brownian motion in an orthant: numerical methods for steady-state analysis.
\textit{Ann. Appl. Probab.} \textbf{2} 65--86

\bibitem{DaMi-11}
J. G. Dai and M. Miyazawa (2011).
Reflecting Brownian motion in two dimensions: exact asymptotics for the stationary distribution.
\textit{Stoch. Syst.} \textbf{1} 146--208

\bibitem{DeZa-05}
J.-D. Deuschel and L. Zambotti (2005).
Bismut-Elworthy's formula and random walk representation for SDEs with reflection. 
\textit{Stochastic Processes Appl.} \textbf{115} 907--925 

\bibitem{DiMo-09}
A. B. Dieker and J. Moriarty (2009).
Reflected Brownian motion in a wedge: sum-of-exponential stationary densities.
\textit{Electron. Commun. Probab.} \textbf{14} 1--16

\bibitem{Du-04}
J. Dub\'edat (2004).
Reflected planar Brownian motions, intertwining relations and crossing probabilities.
\textit{Ann. Inst. H. Poincar\'e Probab. Statist.} 
\textbf{40} 539--552

\bibitem{DuWi-94}
P. Dupuis and R. J. Williams (1994).
Lyapunov functions for semimartingale reflecting Brownian motions.
\textit{Ann. Probab.} \textbf{22} 680--702

\bibitem{ElBeYa-00}
A. El Kharroubi, A. Ben Tahar and A. Yaacoubi (2000).
Sur la récurrence positive du mouvement brownien réfléchi dans l'orthant positif de $\mathbb{R}^n$.
\textit{Stochastics Stochastics Rep.} \textbf{68} 229--253

\bibitem{ErFrHu-20}
P. A. Ernst, S. Franceschi and D. Huang (2021).
Escape and absorption probabilities for obliquely reflected Brownian motion in a quadrant.
\textit{Stochastic Processes Appl.} \textbf{142} 634--670

\bibitem{Fo-84}
M. E. Foddy (1984).
\textit{Analysis of Brownian motion with drift, confined to a quadrant by oblique reflection (diffusions, Riemann-Hilbert problem).}
ProQuest LLC, Ann Arbor, MI. Thesis (Ph.D.)--Stanford University

\bibitem{FoFrIv-22}
V. Fomichov, S. Franceschi and J. Ivanovs (2022).
Probability of total domination for transient reflecting processes in a quadrant.
\textit{Adv. Appl. Probab.} \textbf{54} 1--45

\bibitem{Fr-20}
S. Franceschi (2021).
Green's functions with oblique Neumann boundary conditions in the quadrant. 
\textit{J. Theor. Probab.} \textbf{34} 1775--1810 

\bibitem{FrRa-19}
S. Franceschi and K. Raschel (2019).
Integral expression for the stationary distribution of reflected Brownian motion in a wedge. 
\textit{Bernoulli} \textbf{25} 3673--3713

\bibitem{Ha-78}
J. M. Harrison (1978).
The diffusion approximation for tandem queues in heavy traffic.
\textit{Adv. Appl. Probab.} \textbf{10} 886--905

\bibitem{HaHa-09}
J. M. Harrison and J. Hasenbein (2009).
Reflected Brownian motion in the quadrant: tail behavior of the stationary distribution.
\textit{Queueing Syst.} \textbf{61} 113--138

\bibitem{HaRe-81a}
J. M. Harrison and M. I. Reiman (1981).
On the distribution of multidimensional reflected Brownian motion.
\textit{SIAM J. Appl. Math.} \textbf{41} 345--361

\bibitem{HaRe-81b}
J. M. Harrison and M. I. Reiman (1981).
Reflected Brownian motion on an orthant.
\textit{Ann. Probab.} \textbf{9} 302--308 

\bibitem{HaWi-87}
J. M. Harrison and R. J. Williams (1987).
Multidimensional reflected Brownian motions having exponential stationary distributions.
\textit{Ann. Probab.} \textbf{15} 115--137

\bibitem{HaWi-87bis}
J. M. Harrison and R. J. Williams (1987). 
Brownian models of open queueing networks with homogeneous customer populations. 
\textit{Stochastics} \textbf{22} 77--115

\bibitem{Ha-22}
J. M. Harrison (2022). Reflected Brownian motion in the quarter plane: an equivalence based on time reversal.
\textit{Stochastic Processes Appl.}
\textbf{150} 1189--1203

\bibitem{HoRo-93}
D. G. Hobson and L. C. G. Rogers (1993).
Recurrence and transience of reflecting Brownian motion in the quadrant.
\textit{Math. Proc. Camb. Philos. Soc.} \textbf{113} 387--399

\bibitem{IvBo-15}
J. Ivanovs and O. Boxma (2015). 
A bivariate risk model with mutual deficit coverage. 
\textit{Insurance Math. Econom.} \textbf{64} 126--134

\bibitem{LaRa-13}
P. Lafitte-Godillon, K. Raschel and V. C. Tran (2013). 
Extinction probabilities for a distylous plant population modeled by an inhomogeneous random walk on the positive quadrant. 
\textit{SIAM J. Appl. Math.} \textbf{73} 700--722

\bibitem{LG-87}
J.-F. Le Gall (1987).
Mouvement brownien, c\^ones et processus stables.
\textit{Probab. Theory Related Fields} \textbf{76} 587--627

\bibitem{Le-17}
D. L\'epingle (2017).
A two-dimensional oblique extension of Bessel processes.
\textit{Markov Process. Related Fields} \textbf{23} 233--266

\bibitem{LiRa-19}
D. Lipshutz and K. Ramanan (2019).
Pathwise differentiability of reflected diffusions in convex polyhedral domains.
\textit{Ann. Inst. H. Poincar\'e Probab. Statist.} \textbf{55} 1439--1476

\bibitem{OCOr-14}
N. O'Connell and J. Ortmann (2014).
Product-form invariant measures for Brownian motion with drift satisfying a skew-symmetry type condition.
\textit{ALEA, Lat. Am. J. Probab. Math. Stat.} \textbf{11} 307--329

\bibitem{ReWi-88}
M. I. Reiman and R. J. Williams (1988).
A boundary property of semimartingale reflecting Brownian motions. 
\textit{Probab. Theory Relat. Fields} \textbf{77} 87--97 

\bibitem{Sa-17}
A. Sarantsev (2017).
Reflected Brownian motion in a convex polyhedral cone: tail estimates for the stationary distribution.
\textit{J. Theoret. Probab.} \textbf{30} 1200--1223

\bibitem{Sh-81}
S. E. Shreve (1981).
Reflected Brownian motion in the ``bang-bang'' control of Brownian drift.
\textit{SIAM J. Control Optim.} \textbf{19} 469--478

\bibitem{Ta-90}
L. M. Taylor (1990).
\textit{Existence and uniqueness of semimartingale reflecting Brownian motions in an orthant}. 
Thesis (Ph.D.)--University of California, San Diego

\bibitem{TaWi-93}
L. M. Taylor and R. J. Williams (1993).
Existence and uniqueness of semimartingale reflecting Brownian motions in an orthant. 
\textit{Probab. Theory Relat. Fields} \textbf{96} 283--317 

\bibitem{VaWi-85}
S. R. S. Varadhan and R. J. Williams (1985).
Brownian motion in a wedge with oblique reflection. 
\textit{Comm. Pure Appl. Math.} \textbf{38} 405--443

\bibitem{Wi-85a}
R. J. Williams (1985).
Recurrence classification and invariant measure for reflected Brownian motion in a wedge.
\textit{Ann. Probab.} \textbf{13} 758--778

\bibitem{Wi-85b}
R. J. Williams (1985).
Reflected Brownian motion in a wedge: semimartingale property.
\textit{Z. Wahrsch. Verw. Gebiete} \textbf{69} 161--176

\bibitem{Wi-87}
R. J. Williams (1987).
Reflected Brownian motion with skew symmetric data in a polyhedral domain.
\textit{Probab. Theory Relat. Fields} \textbf{75} 459--485

\bibitem{Wi-95}
R. J. Williams (1995).
Semimartingale reflecting Brownian motions in the orthant.
\textit{Stochastic networks}, 125--137, IMA Vol. Math. Appl., 71, Springer, New York
\end{thebibliography}
\end{document}